\documentclass[letterpaper, 10 pt, conference]{ieeeconf} 
\pdfoutput=1

\IEEEoverridecommandlockouts
\makeatletter                   
\let\proof\@undefined
\let\endproof\@undefined
\makeatother                    
\usepackage[USenglish]{babel}
\usepackage{amsmath} 
\usepackage{amssymb}  
\usepackage{amsthm}
\usepackage{graphicx}
\usepackage{subfig}
\usepackage{mathptmx} 
\usepackage{times} 

\newtheorem{defn}{Definition}
\newtheorem{ass}{Assumption}

\newtheorem{thm}{Theorem}
\newtheorem{prop}{Proposition}

\newtheorem{rem}{Remark}

\newcommand{\R}{\mathbb{R}}
\newcommand{\N}{\mathbb{N}}
\newcommand{\PP}{\mathbb{P}}
\newcommand{\PE}{\mathbb{E}}

\newcommand{\X}{\mathbb{X}}
\newcommand{\Z}{\mathbb{Z}}

\newcommand{\V}{\mathbb{V}}
\newcommand{\W}{\mathbb{W}}

\title{
  An Improved Constraint-Tightening Approach for Stochastic MPC
}
\author{Matthias Lorenzen$^{1}$,
Frank Allg\"ower$^{1}$,
Fabrizio Dabbene$^{2}$,
and Roberto Tempo$^{2}$
\thanks{$^{1}$Institute for Systems Theory and Automatic Control, University of Stuttgart, 
Germany
{\tt\small \{matthias.lorenzen,\allowbreak frank.allgower\}@ist.uni-stuttgart.de}\newline
$^{2}$CNR-IEIIT, Politecnico di Torino, 
Italy
{\tt\small \{roberto.tempo,\allowbreak fabrizio.dabbene\}@polito.it}
}%
}

\begin{document}
\maketitle
\thispagestyle{empty}
\pagestyle{empty}

\begin{abstract}
  The problem of achieving a good trade-off in Stochastic Model Predictive Control between the competing goals of improving the average performance and reducing conservativeness, while still guaranteeing recursive feasibility and low computational complexity, is addressed. 
  We propose a novel, less restrictive scheme which is based on considering stability and recursive feasibility separately. 
  Through an explicit first step constraint we guarantee recursive feasibility. In particular we guarantee the existence of a feasible input trajectory at each time instant, but we only require that the input sequence computed at time $k$ remains feasible at time $k+1$ for most disturbances but not necessarily for all, which suffices for stability.
  To overcome the computational complexity of probabilistic constraints, we propose an offline constraint-tightening procedure, which can be efficiently solved via a sampling approach to the desired accuracy. 
  The online computational complexity of the resulting Model Predictive Control (MPC) algorithm is similar to that of a nominal MPC with terminal region. 
  A numerical example, which provides a comparison with classical, recursively feasible Stochastic MPC and Robust MPC, shows the efficacy of the proposed approach.
%
\end{abstract}

\section{Introduction}
  Stochastic Model Predictive Control (SMPC) formulations, although being computationally much harder than their robust counterpart, have become increasingly popular due to their improved performance and increased region of attraction. A probabilistic description of the disturbance or uncertainty allows to optimize the average performance and to reduce conservatism compared to robust schemes, through allowing a (small) probability of constraint violation. Still, hard constraints, e.g. due to physical limitations, can be considered in the same setup. 

  In Model Predictive Control recursive feasibility, which is essential for stability, is usually guaranteed through showing that the planned input trajectory remains feasible in the next optimization step. In Robust MPC this is done through guaranteeing that the input trajectory remains feasible for all possible disturbances.
  In Stochastic MPC a certain probability of future constraint violation is usually allowed. This leads to significantly less conservative constraint tightening for the predicted input and state because worst case scenarios become very unlikely. However, the probability distribution of the state prediction at some future time depends on both the current state and the time to go. Hence, even under the same control law the violation probability changes from time $k$ to time $k+1$ and might render the optimization problem infeasible.
  The second difficulty unique to SMPC is to render the chance constraints -- constraints on multivariate integrals -- computationally tractable without becoming overly conservative.
  Finally, for nonlinear systems the uncertainty propagation becomes another nontrivial difficulty.

  Significant progress to rigorously addressing the first problem has been done in \cite{Kouvaritakis2010_ExplicitUseOfProbConstr,Cannon2011_StochasticTubesinMPC} where ``recursively feasible probabilistic tubes'' for constraint tightening are proposed. Instead of considering the probability distribution $\ell$ steps ahead given the current state, the probability distribution $\ell$ steps ahead given \emph{any} realization in the first $\ell-1$ steps is considered. This essentially leads to a constraint tightening with $\ell-1$ worst-case and one stochastic prediction for each prediction time $\ell$.
  In \cite{Korda2011_StronglyFeasibleSMPC} the authors propose to compute a control invariant region and to restrict the next state to be inside this region. This procedure leads to a feasible region which is least restrictive, given the affine feedback structure in the MPC control law, but stability issues are not discussed.
  In \cite{Bernandini2013_ScenarioBasedMPCofStochConstrLinSys} this problem is circumvented through optimizing the average performance but considering worst case constraint satisfaction.

  The second problem, tractability of chance constraints, has gained more attention and different methods of relaxation have been proposed in the MPC literature.
  For linear systems with additive stochastic disturbance, the system is usually decomposed into a deterministic, \emph{nominal} part and an autonomous system involving only the uncertain part.
  The approaches can then be divided into (i) computing a confidence region for the uncertain part and using this for constraint tightening, see \cite{Cannon2011_StochasticTubesinMPC} for an ellipsoidal confidence region, and (ii) direct constraint tightening given the evolution of the uncertain part, e.g. \cite{Kouvaritakis2010_ExplicitUseOfProbConstr} and \cite{Korda2011_StronglyFeasibleSMPC}.
  A slightly different approach is taken in \cite{Zhang2013_SochasticMPC}, where the authors first determine a confidence region for the disturbance sequence, as well, but then employ robust optimization techniques.
  For linear systems with parametric uncertainty, \cite{Cheng2014_SMPCforSysWithMultAndAddDist} proposes to decompose the uncertainty tube into a stochastic part offline and a robust part which is computed online. The paper \cite{Fleming2014_StochasticTubeMPCforLPVSysWithSetInclCond} computes online a stochastic tube of fixed complexity using a sampling technique, but which leads to solving a mixed integer problem online.
  In \cite{Kanev2006RobustMPC, Calafiore2013StochasticMPC} the authors use an online sampling approach to cope with the chance constraint and determine in each iteration an optimal feedback gain respectively feed-forward input. While this approach allows for nearly arbitrary uncertainty in the system, the online optimization effort increases dramatically and recursive feasibility cannot be guaranteed. In \cite{Zhang2014_OnSampleSizeOfRandMPCwAppl} the authors use an online sampling approach as well and show how the number of samples can be decreased significantly.
  For nonlinear systems the problem of uncertainty evolution has recently been addressed in \cite{Mesbah2014_StochasticNonlinearMPC} using polynomial chaos expansion.

  The main contribution of this paper is to propose a Stochastic MPC scheme which combines the advantages of the least restrictive approach in \cite{Korda2011_StronglyFeasibleSMPC} and the stability of \cite{Cannon2011_StochasticTubesinMPC}.
  Unlike previous works, we explicitly allow the case when the optimized input sequence does not remain feasible in the next time instance -- but only up to a desired probability $\epsilon_f$.
  With $\epsilon_f=1$ the least restrictive scheme of \cite{Korda2011_StronglyFeasibleSMPC} and with $\epsilon_f=0$, SMPC with recursively feasible probabilistic tubes are recovered. Already for small values of $\epsilon_f$ a significant increase of the feasible region is gained.
  Recursive feasibility is guaranteed through an additional constraint on the first step.
  The resulting offline chance constrained programs are solved efficiently to the desired accuracy using a sampling approach.

  The remainder of this paper is organized as follows. Section \ref{sec:ProbSetup} introduces the receding horizon problem to be solved. In Section \ref{sec:MainRes} the main results are stated, starting with a suitable constraint reformulation, followed by comments on offline solution of the involved chance constraint problems, recursive feasibility of the receding horizon optimization and finally the complete MPC algorithm. Numerical examples that underline the advantages of the proposed scheme are given in Section \ref{sec:NumExample}. Finally Section \ref{sec:Concl} provides some final conclusions and directions for future work. 
\paragraph*{Notation}
The notation employed is standard. Uppercase letters are used for matrices and lower case for vectors. $[A]_j$ and $[a]_j$ denote the $j$-th row and entry of the matrix $A$ and vector $a$, respectively.
Positive (semi)definite matrices $A$ are denoted $A \succ 0$ ($A \succeq 0$) and  $\|x\|_A^2 = x^\top A x$
$\N_+$ denotes the positive integers and ${\N_0 = \{0\} \cup \N_+}$.
We use $x_k$ for the (real, measured) state at time $k$ and $x_{l|k}$ for the state predicted $l$ steps ahead at time $k$.

\section{Problem Setup} \label{sec:ProbSetup}
Consider the linear time-invariant system with state $x_k \in \R^n$, control input $u_k \in \R^{m}$ and additive stochastic disturbance $w_k \in \W \subset \R^{m_w}$
\begin{equation}
  x_{k+1} = A x_k + B u_k + B_w w_k.
  \label{eqn:xsystem}
\end{equation}
In the following we assume that $w_k$ for $k=0,1,2,\ldots$ are independent realizations of a real valued random variable $W$ with 
realizations in $\W$.
Furthermore we assume that $W$ has zero mean and finite variance. 
$\W$ is assumed to be convex (or a convex outer approximation is given) and bounded to include the case of hard constraints.

The system is subject to probabilistic constraints on the state and hard constraints on the input
\begin{subequations}
  \begin{align}
    \PP_k\{ [H]_j x_{k} &\le [h]_j \} \ge 1-[\varepsilon]_j \quad & j\in[1,p],~ k \in \N_+\label{eqn:probConstraints} \\
    G u_{k} &\le g  \quad &k \in \N_0 \label{eqn:inputConstraints}
  \end{align}
  \label{eqn:origConstraints}
\end{subequations}
\hspace{-0.2cm} 
with $H \in \R^{p\times n}$, $G \in \R^{q\times m}$, $h \in \R^{p}$, $g \in \R^{q}$, $\varepsilon \in [0,1)^p$ and $\PP_k$ the $k$-fold product probability measure of $\PP$. 
Equation~\eqref{eqn:probConstraints} restricts to $[\varepsilon]_j$ the probability of violating state constraint $j$ at the future time $k$, given the probability measure of the disturbance sequence $w_0, \ldots, w_{k-1}$ and the current state $x_0$.

The control objective is to determine a receding horizon control, which (approximately) minimizes $J_\infty$, the expected value of an infinite horizon quadratic cost
\begin{equation}
  J_\infty = \lim_{t \rightarrow \infty} \frac{1}{t} \sum_{i=0}^t \PE \left\{ x_i^\top Qx_i + u_i^\top Ru_i \right\}
  \label{eqn:infHorizonCost}
\end{equation}
with $Q\in \R^{n\times n}$, $Q \succ 0$, $R\in \R^{m\times m}$, $R \succ 0$.

\subsection*{Receding Horizon Optimization}
Throughout this paper a standard MPC receding horizon approach with tightened constraints and terminal region is considered, cf. \cite{Rawlings2009_MPC}.

To cope with the state prediction under uncertainty the predicted state $x_{l|k}=z_l+e_l$ of the system is split into a deterministic, nominal part $z_l$ and a (stochastic) error part $e_l$. A prestabilizing error feedback $u_e = K e$ is employed which leads to the predicted input $u_{l|k} = Ke_l + v_l$ with $v_l$ the free MPC input. The system description of the predicted nominal state and error is given by
\begin{subequations}
  \begin{align}
    z_{l+1} &= A z_l + B v_l  &z_0 = x_k\label{eqn:zSys}\\
    e_{l+1} &= A_{cl}e_l + B_w w_l  &e_0 = 0 \label{eqn:eSys}
  \end{align}
  \label{eqn:zAndeSys}
\end{subequations}
with $A_{cl} = (A+BK)$.

The finite horizon cost $J_T(x_k,u_{[0|k, T-1|k ]})$ to be minimized at time $k$ is defined as
\begin{multline}
    J_T(x_k,u_{[0|k, T-1|k ]}) \\
    = \PE\left\{ \sum_{l=0}^{T-1} \left( x_{l|k}^\top Qx_{l|k} + u_{l|k}^\top Ru_{l|k} \right) + x_{T|k}^\top P x_{T|k}\right\}
  \label{eqn:finiteHorizonCostFnc}
\end{multline}
where $P$ is the solution to discrete-time Lyapunov equation $A_{cl}^\top P A_{cl}+  Q + K^\top R K^\top = P$.
The expected value can be solved explicitly which gives a quadratic, finite horizon cost function in the deterministic variables $z_0$ and $v$
\begin{equation}
  J_T(z_0,v_{[0,T-1]}) = \sum_{i=0}^{T-1} \left( z_i^\top Qz_i + v_i^\top Rv_i \right) + z_T^\top P z_T + c
  \label{eqn:finiteHorizonCostFncDet}
\end{equation}
where $c = \PE \left\{ \sum_{i=0}^{T-1} e_i^\top (Q +K^\top RK)e_i + e_n^\top Pe_n \right\}$ is a constant term and can be neglected in the optimization problem.

The full optimization problem is now stated, where $\Z_{l}$ and $\V_{l}$ are suitable constraints derived from~\eqref{eqn:origConstraints} and some terminal constraint.
\begin{defn}[Finite Horizon Optimization Problem]
  \begin{equation}
    \begin{aligned}
      \min_{v_0,\ldots,v_{T-1}} ~& J_T(z_0,v_{[0,T-1]})\\
      \text{s.t.} ~& z_{l+1} = A z_l + B v_l,  \quad z_0 = x_k \\
      ~& z_l \in \Z_{l}, ~ l\in[1,T]\\
      ~& v_l \in \V_{l}, ~ l\in[0,T-1].
    \end{aligned}
    \label{eqn:RecHorizonOptProg}
  \end{equation}
\end{defn}

\section{Main Results} \label{sec:MainRes}
In the following we address three of the main problems in SMPC. Namely, how to generate computationally tractable, nonconservative constraint sets $\Z_l$ and $\V_l$ such that (i) the constraints~\eqref{eqn:origConstraints} hold in closed loop operation, (ii) if initially feasible, the optimization remains feasible under all admissible disturbance sequences, and (iii) the closed loop system is stable (in a suitable sense).

\subsection{Constraint Tightening}
  Similar to \cite{Kouvaritakis2010_ExplicitUseOfProbConstr} we directly tighten the constraints offline. But in contrast we neither aim at recursively feasible probabilistic tubes nor at robust constraint tightening for the input.

\subsubsection*{State Constraints}
The probabilistic state constraints~\eqref{eqn:probConstraints} can be rewritten in terms of hard constraints $\Z_l$ on the predicted nominal state $z_{l}$.
\begin{prop} \label{prop:constrSatisf}
  If the nominal system~\eqref{eqn:zSys} satisfies the constraints 
  \begin{equation}
    H z_l \le \eta_{l} \quad l \in [1, T-1]
    \label{eqn:detStateConstraint}
  \end{equation}
  with
  \begin{equation}
    \begin{aligned}
      {[\eta_{l}]_j} = \max_{\eta} &~ \eta\\
      \text{s.t.} &~ \PP_{l}\left\{ \eta \le [h]_j -  [H]_j  e_l \right\} \ge 1- [\varepsilon]_{j}
    \end{aligned}
    \label{eqn:offlineChanceConstrProgr}
  \end{equation}
  for $l=1,\ldots,T-1$ and $j=1,\ldots,p$, then the real system~\eqref{eqn:xsystem} satisfies the chance constraints~\eqref{eqn:probConstraints} for $k=1,\ldots,T-1$ and $j=1,\ldots,p$.
\end{prop}
\begin{proof}
The constraint~\eqref{eqn:probConstraints} can be rewritten in terms of $z_l$ and $e_l$ as
\begin{equation*}
  \PP_l\left\{  [H]_j z_l \le [h]_j -  [H]_j e_l \right\} \ge 1-[\varepsilon]_{j}
\end{equation*}
with $e_l$ being the solution to~\eqref{eqn:eSys}. Given $[H]_j z_l \le [\eta_{l}]_j$ and $\PP_{l}\left\{ [\eta_l]_j \le [h]_j -  [H]_j  e_l \right\} \ge 1- [\varepsilon]_{j}$ it follows that $\PP_{l}\left\{ [H]_j z_l \le [h]_j -  [H]_j  e_l \right\} \ge 1- [\varepsilon]_{j}$.
\end{proof}
Proposition~\ref{prop:constrSatisf} leads to $(T-1)p$ one dimensional, linear chance constrained optimization problems~\eqref{eqn:offlineChanceConstrProgr} that need to be solved offline. Efficient methods will be presented in the next subsection.

\subsubsection*{Input Constraints}
Instead of a robust constraint tightening for the hard constraints on the input $u$, we propose a stochastic constraint tightening as well. In other words, we take advantage of the probabilistic nature of the disturbance and require that the combination of MPC feedforward input sequence and static error feedback remains feasible for most, but not for all possible disturbance sequences. This is in line with the fact that at a later time the optimal input is recomputed and adapted to the disturbance realization.

Let $\epsilon_u \in [0,1)$ be a small probability. Similarly to the state constraint tightening, we replace the original constraint~\eqref{eqn:inputConstraints} with 
\begin{equation}
  G v_l \le \mu_l \quad l \in [0, T-1]
  \label{eqn:detInputConstraint}
\end{equation}
where again $\mu_l = [\mu_{l1} \ldots \mu_{lq}]$ are the solutions to $T q$ one dimensional, linear chance constrained optimization problems
\begin{equation}
  \begin{aligned}
    {[\mu_l]_j} = \max_{\mu} &~ \mu\\
    \text{s.t.} &~ \PP_{l}\left\{ \mu \le [g]_j -  [G]_j  K e_l \right\} \ge 1- \varepsilon_u.
  \end{aligned}
  \label{eqn:offlineChanceConstrProgrU}
\end{equation}


\subsubsection*{Terminal Constraint}
We first construct a recursively feasible admissible set under the local control law and then employ a suitable tightening to determine the terminal constraint $\Z_T$ for the nominal system.
\begin{prop}
  For the system~\eqref{eqn:xsystem} with input $u=Kx$ let $\X_T=\{H_T x \le h_T\}$ be a (maximal) robust positive invariant polytope inside
  \begin{equation*}
    \tilde \X_T = \left\{ x ~|~ H A_{cl} x \le \eta_{1},~ G K x \le g \right\}
  \end{equation*}
  with $\eta_1 = [\eta_{11}, \eta_{12}, \ldots, \eta_{1p}]$ according to~\eqref{eqn:offlineChanceConstrProgr}.
  For any initial condition in $\X_T$ the constraints~\eqref{eqn:origConstraints} are satisfied in closed loop operation with the local control law $u=Kx$ for all $k \ge 0$.
\end{prop}
\begin{proof}
  By definition the set $\X_T$ is forward invariant for all disturbances and
  the constraints
  \begin{equation*}
    \PP_k\{ [H]_j x_{k} \le [h]_j ~|~ x_{k-1} \} \ge 1-[\varepsilon]_j \quad \forall j\in[1,p]
  \end{equation*}
  are satisfied for all states $x_{k-1} \in \X_T$, which is sufficient for~\eqref{eqn:probConstraints}.
\end{proof}
For an in depth theoretical discussion, practical computation and polytopic approximations of $\X_T$ see \cite{Blanchini1999_SetInvarianceInControl} for an overview or \cite{Kolmanovsky1998_TheoryAndComputationOfDisturbanceInvariantSets} for details.

To define the terminal constraint $\Z_T$ for the nominal system, a constraint tightening approach similar to~\eqref{eqn:offlineChanceConstrProgr} is necessary. Let $\epsilon_T \in [0,1)$ be a small probability, we define the terminal region
  \begin{equation}
    \Z_T = \{ z ~|~ H_T z \le \eta_T \}
    \label{eqn:termConstr}
  \end{equation}
  with
 \begin{equation}
  \begin{aligned}
    {[\eta_T]_j} = \max_{\eta} &~ \eta\\
    \text{s.t.} &~ \PP_{l}\left\{ \eta \le [h_T]_j -  [H_T]_j  K e_T \right\} \ge 1- \varepsilon_T.
  \end{aligned}
  \label{eqn:offlineChanceConstrProgrTermConstr}
\end{equation} 

\subsection{Solving the Single Chance Constrained Programs}
There is a vast literature on how to solve optimization programs involving single chance constraints. In the following we briefly state the deterministic solution and then show how to efficiently solve the offline problems~\eqref{eqn:offlineChanceConstrProgr}, \eqref{eqn:offlineChanceConstrProgrU} and \eqref{eqn:offlineChanceConstrProgrTermConstr} using a sampling approach.
\subsubsection{Deterministic}
Chance constraints are constraints on multivariate integrals, in particular, if the random variable $W$ has a known probability density function $f_W(w)$ we can write~\eqref{eqn:offlineChanceConstrProgr} as
\begin{equation*}
  \begin{aligned}
      {[\eta_{l}]_j} = \max_{\eta} &~ \eta\\
    \text{s.t.} &~ \int_{\W^l} \mathbf{1} \left\{ \eta \le [h]_j - \left[ H \right]_j  e_l  \right\} \\
    &\hspace{2.0cm} \prod_{i=0}^{l-1}f_W(w_i) d w_0 \cdots dw_{l-1} \ge 1-[\varepsilon]_j
  \end{aligned}
\end{equation*}
with $e_l = \sum_{i=0}^{l-1}A_{cl}^iB_w w_i$ and $\mathbf{1}\{\cdot\}$ being the indicator function.
The multivariate integral can be further simplified if the convolution of the distributions of $B_w w_i$ is known, e.g. $W$ is Gaussian. If even the inverse cumulative distribution function $Q$ of $[ H ]_j e_l$ is known or can be approximated, e.g. $W$ being Gaussian, then
\begin{equation*}
  {[\eta_l]_j} = [h]_j - Q(1-[\varepsilon]_j).
\end{equation*}
For further discussion on convexity and explicit numerical solution cf. \cite{prekopa2010_StochasticProgramming} and references therein.

\subsubsection{Sampling}
Recently, sampling techniques to solve robust and chance constrained problems have gained increased interest, see \cite{Tempo2012_RandAlgForAnalysisAndDesign,Calafiore2011_ProbMethodsFrContrSysDesign} for further discussion about randomized algorithms. They are easy to implement and specific guarantees about their solution can be given. Furthermore, they allow to directly use complicated simulations or measurements of the error instead of determining a probability density function. A linear relation $B_w w$ is not necessary but $B_w(w)$ can be assumed instead. 

The chance constrained problem~\eqref{eqn:offlineChanceConstrProgr}, \eqref{eqn:offlineChanceConstrProgrU}, \eqref{eqn:offlineChanceConstrProgrTermConstr} can as well be efficiently solved to the desired accuracy by drawing a sufficiently large number $N_s$ of samples $w^{(i)}$ from $W$ and require the constraint to hold for all, but a fixed number $r$ of samples.
In \cite{Campi2011_SampleAndDiscardApprTpCCOpt, Calafiore2010_RandomConvexPrograms} the authors give explicit bounds on how to choose $N_s$ and $r$ such that the optimal solution of the sampled problem has the desired accuracy with an a priori specified confidence.

In general, one has to solve a mixed integer problem or to use heuristics to discard samples in an optimal way. Here, due to the simple structure, a sort algorithm is used to solve the sampled approximation of~\eqref{eqn:offlineChanceConstrProgr}.
\begin{prop}
  Let $N_s$ and $r$ be chosen according to \cite{Campi2011_SampleAndDiscardApprTpCCOpt} for an accuracy $\alpha$ and confidence $\beta$.
  Let $q_{1-r/N_s}$ be the $(1-r/N_s)$-quantile of the set $\left\{ [H]_j e_l^{(i)}\right\}_{i=1,\ldots,N_s}$ with ${e_l^{(i)} = \sum_{j=1}^l A_{cl}^{j-1}B_w w_j^{(i)}}$ independently chosen samples from $W^l$.
  Then with confidence $\beta$
  \begin{equation*}
    {[\eta_l]_j} = [h]_j - q_{1-r/N_s}
  \end{equation*}
  solves~\eqref{eqn:offlineChanceConstrProgr} with an accuracy $\alpha$.
\end{prop}

\subsection{Recursive Feasibility}
As it has been pointed out in previous publications, e.g. \cite{Kouvaritakis2010_ExplicitUseOfProbConstr,Cannon2011_StochasticTubesinMPC}, the probability of constraint violation $\ell$ steps ahead at time $k$ is not the same as $\ell-1$ steps ahead at time $k+1$ given the realization of state $x_{k+1}$. Hence, the tightened constraints~\eqref{eqn:detStateConstraint}, \eqref{eqn:termConstr} and \eqref{eqn:detInputConstraint} do not guarantee recursive feasibility.

A commonly used approach to recover recursive feasibility is to use a mixed worst-case/stochastic prediction for constraint tightening. In \cite{Korda2011_StronglyFeasibleSMPC} the authors pointed out that this approach guarantees recursive feasibility, but is rather restrictive and leads to higher average costs if the optimal solution is ``near'' a chance constraint. The authors propose to use a first step constraint to obtain a recursively feasible algorithm which is, given the affine feedback structure in the MPC, least restrictive.

In the following we propose a hybrid strategy: We impose a first step constraint to guarantee recursive feasibility \emph{and} the previously introduced stochastic tube tightening with terminal constraint and cost to guarantee stability. At the cost of further offline reachability and controllability set computation, the proposed approach has the advantage of being less conservative, but yet guaranteed to stabilize the system at the minimal positive invariant region.

Let
\begin{equation*}
  C_T = \left\{ \begin{bmatrix} z_0 \\ v_0  \end{bmatrix} \in \R^{n+m} :
    \begin{array}[h]{l}
      \exists v_1, \ldots, v_{T-1} \in \R^m\\
      z_{l+1} = A z_l + B v_l \\
      H z_l \le \eta_{l}, ~ l\in[1,T-1]\\
      G v_l \le \mu_{l}, ~ l\in[0,T-1]\\
      H_T z_T \le \eta_T
    \end{array}
\right\}
\end{equation*}
be the $T$-step set and allowed first step input for the nominal system under the tightened constraints. The set can be computed via standard recursion e.g. \cite{Gutman1987_AlgorithmToFindMaximalStateConstrSet}. $ C_T$ defines the feasible states and first inputs of the receding horizon optimization.

Since $ C_T$ is not necessarily robust positive invariant, it is important to further compute a (maximal) robust control invariant polytope $ C_T^\infty$ inside $ C_T$. This again can be computed via standard recursions, for algorithms and their finite termination cf. \cite{Blanchini1999_SetInvarianceInControl} and references therein.
For convenience we define $C_{T,x}^\infty = \operatorname{Proj}_x(C_T^\infty)$ to be the projection of $C_T^\infty$ onto the first $n$ coordinates.

\begin{ass}
  The set $C_{T,x}^\infty$ is bounded.
\end{ass}

\begin{rem}
  It is important to keep the constraint on the input $v_0$ in the computation of the robust control invariant set in order to guarantee existence of an input that makes the set robustly forward invariant \emph{and} steers the nominal system into the terminal region.
\end{rem}

\begin{rem}
  The computation of the sets $ C_T$ and $ C_T^\infty$ may be involved for high dimensions and limits the proposed approach. Nevertheless, this is a long-standing, standard problem in (linear) controller design and efficient algorithms to exactly calculate or approximate those sets exist, e.g. \cite{Kolmanovsky1998_TheoryAndComputationOfDisturbanceInvariantSets}.
\end{rem}

\subsection{Resulting Stochastic MPC Algorithm}
The final MPC algorithm can be divided into two parts: (i) an offline computation of the involved sets and (ii) the repeated online optimization. In the following, we present the algorithm and state its control theoretic properties.
\\
\emph{Offline:} Solve~\eqref{eqn:offlineChanceConstrProgr}, \eqref{eqn:offlineChanceConstrProgrU} and \eqref{eqn:offlineChanceConstrProgrTermConstr} to determine $\eta_{l}$, $\mu_{l}$ for ${l=0 \ldots T}$.
Determine the first step constraint $ C_T^\infty$ according to the previous section.
\\
\emph{Online:} For each time step $k=1,2,\ldots$
\begin{enumerate}
  \item Measure current state $x_k$,
  \item Solve the linearly constrained quadratic program~\eqref{eqn:RecHorizonOptProg} subject to state and input constraints~\eqref{eqn:detStateConstraint}, \eqref{eqn:detInputConstraint}, first step constraint $ C_T^\infty$ and terminal constraint~\eqref{eqn:termConstr}.
    \begin{equation}
      \begin{aligned}
        v_0,\ldots,v_{T-1} ~&= \arg \min_{v_0,\ldots,v_{T-1}} J_T(x_k,v_{[0,T-1]})\\
        \text{s.t.} ~& z_{l+1} = A z_l + B v_l  \quad z_0 = x_k \\
        ~& H z_l \le \eta_{l}, ~ l\in[1,T-1]\\
        ~& G v_l \le \mu_{l}, ~ l\in[0,T-1]\\
        ~& H_T z_T \le \eta_T \\
        ~& (z_0,v_0) \in  C_T^\infty,
      \end{aligned}
      \label{eqn:MPCOptProg}
    \end{equation}
  \item Apply $v_0$
\end{enumerate}

\begin{prop} \label{prop:recFeas}
  The MPC optimization remains feasible if the initial state is inside $C_{T,x}^\infty$.
\end{prop}
\begin{proof}
  Since $C_{T,x}^\infty$ is a subset of the feasible set, a solution to~\eqref{eqn:MPCOptProg} exists for all $x_k \in C_{T,x}^\infty$. Furthermore since $(x_k,v_0) \in  C_T^\infty$ it holds that $x_{k+1} \in C_{T,x}^\infty$ because of the robust forward invariance property.
\end{proof}

Due to the persistent excitation, it is clear that the system will not converge asymptotically to the origin, but will ``oscillate'' with bounded variance around it.
\begin{thm}
  If $x_0 \in C_{T,x}^\infty$ then the closed loop system under the proposed MPC control law satisfies the probabilistic constraint~\eqref{eqn:probConstraints} for all future times and 
  \begin{equation*}
    \lim_{t\rightarrow \infty} \frac{1}{t} \sum_{k=0}^t \PE\left\{ \|x_k\|_Q^2 \right\} \le (1-\epsilon_f) \PE\left\{ \|B_w w\|_P^2 \right\} + \epsilon_f C
  \end{equation*}
  with $\epsilon_f$ the maximum probability that the previously planned trajectory does not remain feasible, $C = L~max_{w\in\W}\|B_w w\|$ and $L$ the Lipschitz constant of the optimal value function $J_T(\cdot,v_{[0,T-1]}^*)$ of~\eqref{eqn:MPCOptProg}.
\end{thm}
\begin{proof}
  Chance constraint satisfaction follows from Proposition~\ref{prop:constrSatisf} and hard input constraint satisfaction from ${e_0 = 0}$ and hence $\mu_0 = g$. Recursive feasibility follows from Proposition~\ref{prop:recFeas}.

  To prove the second part, we use the optimal value of~\eqref{eqn:MPCOptProg} as a stochastic Lyapunov function.
  Let $V(x_k) = J(x_k,v^*_{[0,T-1]})$ be the 
  optimal value of~\eqref{eqn:MPCOptProg} at time $k$.
  The optimal value function is known to be continuous, convex and piecewise quadratic in $x_k$ \cite{Bemporad2002_ExplicitLQRforConstrSys}, hence a Lipschitz constant $L$ on $C_{T,x}^\infty$ exists.
  The old input trajectory does not remain feasible with at most probability $\epsilon_f$, but we can bound the cost increase in that case by $L~\max_{w\in\W} \|B_w w\|$.
  
  Let $\PE\left\{ V(x_{k+1}) | x_k, v_1^*,\ldots,v_{T}^*~\text{feasible}\right\}$ be the expected optimal value at time $k+1$, conditioning on the state at time $k$ and feasibility of the previously optimal input trajectory
  \begin{equation*}
    \begin{aligned}
      & \PE\left\{ V(x_{k+1}) | x_k, v_1^*,\ldots,v_{T}^*~\text{feasible} \right\} - V(x_k) \\
      =& \sum_{l=1}^{T-1} \left(  \|z_l^*\|_Q^2 + \|v_l^*\|_R^2  \right) + \|z_T^*\|_{(Q+K^\top R K)}^2 + \|z_{T+1}^*\|_P^2 ~ \\
      & + \PE\left\{ \sum_{l=1}^{T} \| A_{cl}^{l-1}B_w w_k\|_{(Q +K^\top RK)}^2 + \|A_{cl}^{T}B_w w_k \|_P^2 \right\} \\
      & -\left( \sum_{l=0}^{T-1} \left(  \|z_l^*\|_Q^2 + \|v_l^*\|_R^2  \right) + \|z_T^*\|_P^2 \right) \\
      =& \|z_T^*\|_{(Q+K^\top R K)}^2 + \|z_{T+1}^*\|_P^2 - \|z_0^*\|_Q^2 - \|v_0^*\|_R^2  - \|z_T^*\|_P^2  \\
      & + \PE\left\{ \|B_w w\|_P^2 \right\}  \\
      \le& - \|z_0\|_Q^2 + \PE\left\{ \|B_w w\|_P^2 \right\} 
    \end{aligned}
  \end{equation*}
  where $z_0 = x_k$, $e_0 = 0$ and $v_l^*$, $z_{l+1}^*$, $l=1,\ldots, T-1$ denote the optimal solution of~\eqref{eqn:MPCOptProg} at time $k$ and $z_{T+1}^* = (A+BK)z_T^*$. Note that the expected value of all $w$-$z$ cross-terms equals zero because of the zero-mean and independence assumption. Furthermore, since we defined the terminal cost as the solution to the discrete-time Lyapunov equation it holds that $A_{cl}^\top P A_{cl} + Q + K^{\top}RK = P$.
 
  Taking iterated expectations gives
\begin{equation*}
  \begin{aligned}
    & \PE\left\{ V(x_{k+1}) | x_k \right\} - V(x_k) \\
    =& (1-\epsilon_f)\left(\PE\left\{ V(x_{k+1}) | x_k,v_1,\ldots,v_{T}^*~\text{feas.} \right\} - V(x_k)\right) + \\
    & \epsilon_f ~L~\max_{w\in\W} \|B_w w\| \\
    \le & (1-\epsilon_f)\left(-\|z_0\|_Q^2 + \PE\left\{ \|B_w w\|_P^2 \right\}\right) +  \epsilon_f C.
  \end{aligned}
\end{equation*}
  The final statement follows by Dynkin's Formula, cf. Theorem 2.6 in \cite{Kushner1967_StochStabilityControl}.
\end{proof}
\begin{rem}
  The parameter $\epsilon_f$ can be designed similar to the procedure described in \cite{Kouvaritakis2010_ExplicitUseOfProbConstr} where it is (essentially) equal to $1$.
\end{rem}

\begin{rem}
  Instead of a robust forward invariant terminal region, a terminal region, which is forward invariant with probability $\epsilon_f$, can be used without changing the result. Still, for each disturbance the next state should remain inside $C_{T,x}^\infty$.
  In case a robust forward invariant terminal region is used, the even stronger condition
  \begin{equation*}
    \lim_{t\rightarrow \infty} \frac{1}{t} \sum_{k=0}^t \PE\left\{ \| x_k \|_Q^2 \right\} \le \PE\left\{ \| B_w w\|_P^2 \right\}
  \end{equation*}
  holds.
\end{rem}

\section{Numerical Example} \label{sec:NumExample}
In this section, we demonstrate the performance and enlarged region of attraction of the proposed Stochastic MPC scheme.
To this end, we implemented the DC-DC converter system example taken from \cite{Cannon2011_StochasticTubesinMPC}.
The linearized system is of the form~\eqref{eqn:xsystem} with
\begin{equation*}
  A = \begin{bmatrix} 1 & 0.0075 \\ -0.143 & 0.996 \end{bmatrix}, \quad 
  B = \begin{bmatrix} 4.798 \\ 0.115 \end{bmatrix}, \quad 
  B_w = I_2.
\end{equation*}
The MPC cost weights are $Q = \begin{bmatrix} 1 & 0 \\ 0 & 10 \end{bmatrix}$, $R=1$ and the prediction horizon is $T=8$. For disturbance attenuation in the predictions~\eqref{eqn:eSys} and terminal region, the unconstrained LQR is chosen.
The disturbance distribution is assumed to be a truncated Gaussian with the covariance matrix $\Sigma = \frac{1}{25^2}I_2$ truncated at $\|w\|^2 \le 0.02$.

For the robust set calculations in the terminal region and first step constraint, we chose a polytopic outer approximation with 8 hyperplanes. For the stochastic constraint tightening we used the described sampling approach with an accuracy 
such that the sampled problem is equal to the true chance constrained problem with an $\epsilon$ within the range $[0.95\epsilon,1.05\epsilon]$
and confidence $\beta = 10^{-4}$.

\subsection{Constraint Violation}
First, consider the single chance constraint
\begin{equation}
  \PP_k\left\{ x_1 \le 2 \right\} \ge 0.8
  \label{eqn:ex1Constraint}
\end{equation}
for the linearized DC-DC converter system and initial state of $x_0 = [2.5 ~ 2.8]^\top$.

In \cite{Cannon2011_StochasticTubesinMPC} it has been shown that Stochastic MPC achieves lower closed loop cost compared to Robust MPC. The approach presented in \cite{Cannon2011_StochasticTubesinMPC}, using a confidence region, yields $14.4\%$ constraint violation in the first $6$ steps.

In contrast, the approach taken here, i.e. a direct constraint tightening, achieves a closed loop operation tight at the constraint. A Monte Carlo simulation with $10^4$ realizations showed an average constraint violation in the first 6 steps of $20\%$ and an even lower closed loop cost. Simulation results of the closed loop system for $500$ random disturbances are shown in Figure~\ref{fig:SimResBoth}. The left plot shows the complete trajectories for a simulation time of $15$ steps. The right plot shows the constraint violation in more detail, \eqref{eqn:ex1Constraint} is satisfied with the maximal allowed constraint violation and hence best performance.
\begin{figure}[htpb]
  \subfloat{{ 
    \includegraphics[width=0.30\textwidth]{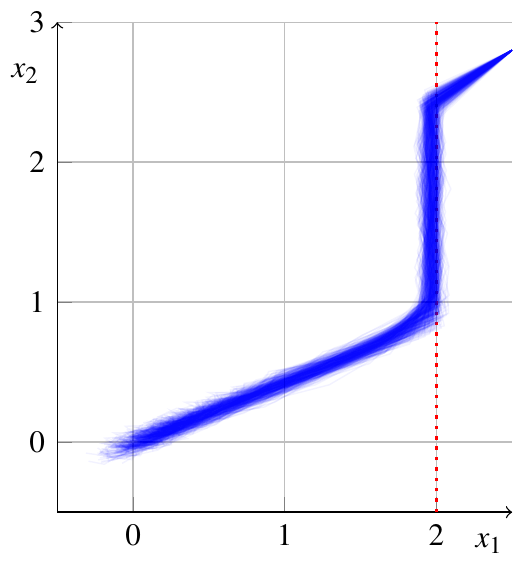}
    }}
  \hspace{-0.3cm}
  \subfloat{{ 
    \includegraphics[width=0.15\textwidth]{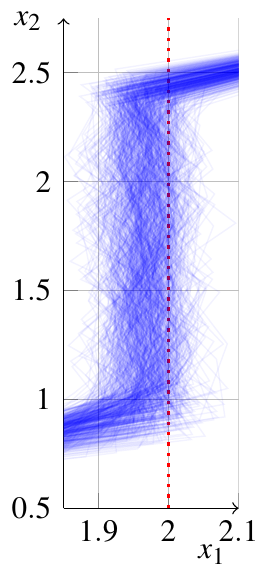}
  }}
  \caption{\small Left: Plot of closed loop response with 500 different disturbance realizations and initial state $x_0 = [2.5 ~ 2.8]^\top$. \\
    Right: Detail showing the trajectories near the constraint $\PP\{x_1 \le 2\} \ge 0.8$.
A Monte Carlo simulation with $10^4$ realizations showed an average constraint violation in the first 6 steps of $20\%$. }
  \label{fig:SimResBoth}
\end{figure}

For comparison, we remark that Robust MPC achieves $0\%$ constraint violation and the $LQ$ optimal solution violates the constraint $100\%$ in the first $3$ steps.

\subsection{Feasible Region}
The main advantage of the proposed Stochastic MPC scheme, compared to more standard use of ``recursively feasible probabilistic tubes'' \cite{Kouvaritakis2010_ExplicitUseOfProbConstr}, is the increased feasible region.

We assume the same setup as before, but with additional chance constraints on the state and hard input constraints
\begin{equation*}
  \begin{aligned}
    \PP_k\{~ |x_1| &\le 2 ~\} \ge 0.8 \\
    \PP_k\{~ |x_2| &\le 3 ~\} \ge 0.8 \\
              |u| &\le 0.2.
  \end{aligned}
\end{equation*}
According to the described setup, we allowed $5\%$ constraint violation in the predictions for the input and a probability of $0.05$ of not reaching the terminal region. In closed loop operation the input was treated as hard constraint.

Figure~\ref{fig:CompFeasReg} shows the different feasible regions of Robust MPC, Stochastic MPC with constraint tightening using recursively feasible probabilistic tubes and the proposed method using probabilistic tubes and a first step constraint.
The feasible region of proposed Stochastic MPC has $1.8$ times the volume of the feasible region of \emph{standard} SMPC and $3.4$ times the volume of the feasible region of Robust MPC. The Robust MPC scheme has been taken from \cite{Mayne2005_RobustMPCofConstrLinSysWithBoundedDist} and only included here for a more complete comparison, it is of course significantly smaller than having stochastic constraints.
\begin{figure}[htpb]
  \begin{center}
    \includegraphics[width=0.45\textwidth]{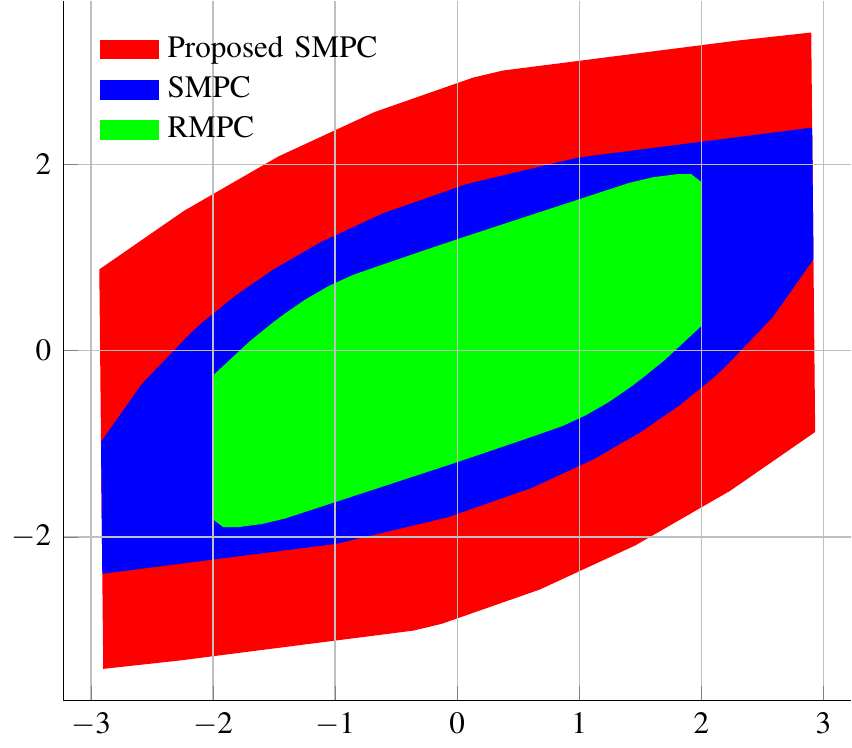}
  \end{center}
  \caption{\small Comparison of feasible region for Robust MPC, Stochastic MPC with recursively feasible probabilistic tubes and proposed Stochastic MPC with guaranteed recursive feasibility.}
  \label{fig:CompFeasReg}
\end{figure}

\section{Conclusions and Further Work} \label{sec:Concl}
The proposed Stochastic MPC algorithm provides a significantly increased feasible region through separating the requirements of recursive feasibility and stability. The stochastic information about the disturbance is used to prove a Lyapunov condition on the average cost. The absolute bounds are used to provide a first step constraint to guarantee recursive feasibility. The online computational effort is equal to that of nominal MPC.
An efficient, broadly applicable solution strategy based on randomized algorithms is presented to solve the offline chance constrained problems to the desired accuracy.

Future work could include choosing appropriate feedback gains to shape the probability distribution of the predicted state in order to better satisfy the constraints. The performance could be improved through an online evaluation of the expected cost, taking into account future infeasibility of the optimized input trajectory.

The idea to incorporate a first step constraint to guarantee recursive feasibility could be further exploited. In the future this could be used in a broader context, e.g. for online sampling to guarantee recursive feasibility in spite of nonzero probability of failure of sampling techniques. It could be nicely combined with ideas of (incomplete) decision trees which show very good results in practice, e.g. \cite{Lucia2013_MultiStageNMPCforSemiBatchReactor}, but have no recursive feasibility or stability guarantees.

Ongoing work includes relaxing the assumption of identically and independently distributed disturbance to e.g. Markov chain models, including parametric uncertainty and implementation and testing at a real-world control problem.

\bibliographystyle{IEEEtranNoUrl}
\bibliography{GlobalBib}
\end{document}